\documentclass[reqno, 12 pt]{amsart}

\usepackage{amsthm}
\usepackage{amssymb}
\usepackage{latexsym}
\usepackage{url}
\usepackage{amsmath}
\usepackage{verbatim}
\usepackage{graphicx}
\usepackage{epsf}
\usepackage{enumerate}
\usepackage{geometry}
\usepackage[width=6cm]{caption}
\usepackage{hyperref}
\hypersetup{breaklinks=true,
pagecolor=white,
colorlinks=true,
citecolor=black,%
filecolor=black,%
linkcolor=black,%
urlcolor=black
}

\newtheorem{thm}{Theorem}[section]

\newtheorem{lem}[thm]{Lemma}

\newtheorem{defin}[thm]{Definition}

\theoremstyle{definition}
\newtheorem{expl}[thm]{Example}
\newtheorem{remark}[thm]{Remark}

\newcommand{\dom}{\textnormal{dom}}

\newcommand{\Int}{\textnormal{int}}
\newcommand{\Ext}{\textnormal{Ext}}

\newcommand{\R}{\mathbb{R}}

\newcommand{\wt}{\widetilde}

\title[Topological properties of sets represented by an inequality]{Topological properties of sets represented by an inequality involving 
distances} 
\author{Daniel Reem}
\address{\noindent IMPA - Instituto Nacional de Matem\'atica Pura e Aplicada, Estrada Dona Castorina 110, Jardim Bot\^anico, CEP 22460-320, Rio de Janeiro, RJ,  Brazil.} 
\email{\noindent dream@impa.br }
\keywords{boundary, closure, interior, strong triangle inequality, 
uniformly convex normed space, Voronoi cell.}
\subjclass[2010]{46B20, 68U05, 46N99, 65D18} 
\date{April 29, 2013}  
\begin{document}
\maketitle
\begin{abstract}
Consider a set represented by an inequality. An interesting phenomenon which 
occurs in various settings in mathematics is that the interior of this set is the 
subset where strict inequality holds, the boundary is the subset where equality holds, 
and the closure of the set is the closure of its interior. 
This paper discusses this phenomenon assuming the set is a Voronoi cell induced 
by given sites (subsets), a geometric object which appears  
in many fields of science and technology and has diverse applications. 
Simple counterexamples show that the discussed phenomenon does not hold in general, 
but it is established in a wide class of cases. More precisely, 
the setting is a (possibly infinite dimensional) uniformly 
convex normed space with arbitrary positively separated sites. 
An important ingredient in the proof is a strong version of 
the triangle inequality due to Clarkson (1936), an interesting inequality which 
has been almost totally forgotten. 
\end{abstract}

\section{\bf Introduction}\label{sec:Introduction}
\subsection{Background: }\label{subsec:Background}  Consider a set represented by an inequality. An intuitive rule of thumb says that its interior 
is the set where strict inequality holds, its boundary is the set where equality holds, 
and the closure of the interior is the closure of the set itself.  
This intuition probably comes from familiar and simple examples in $\R^n$  such as 
balls, halfspaces, and polyhedral sets, or the ones described in \cite[p. 192]{Fulks1961},\cite[p. 6]{Lawson2003}. 
Another well known example is the case of level sets of  convex functions. 
Given a convex function $f:\R^n\to\R$, denote 
its so-called 0-level set by 
\begin{equation}\label{eq:0Level}
S:=f^{\leq 0}:=\{x\in\R^n: f(x)\leq 0\}.
\end{equation}
If the so-called Slater's condition 
 holds \cite[p. 325]{Bertsekas1999},\cite[p. 44]{BorweinLewis2006}, \cite{Slater1950},  \cite[p. 98]{VanTiel1984}, 
namely that $f(x_0)<0$ for some $x_0\in \R^n$, then the  interior of $S$ 
is $f^{<0}:=\{x\in\R^n: f(x)<0\}$. See, for instance, \cite[p. 59]{Rockafellar1970}. 
This property can be easily generalized to any topological vector space when $f$ is assumed to be 
continuous \cite[pp. 80, 117]{VanTiel1984}. 
An additional closely related well-known result says that the closure of a convex 
set $C$ whose (relative) interior  is nonempty is the (relative) closure of the (relative) interior of the set 
\cite[pp. 8-9]{BorweinLewis2006},\cite[p. 114]{KellyWeiss1979},\cite[p. 46]{Rockafellar1970} ($\R^n$),  
\cite[pp. 29-30]{VanTiel1984} (topological vector space). This property can be expressed 
as $\overline{C}=f^{\leq 1}=\overline{f^{<1}}$,  
where $f(x)=p(x-x_0)$, $x_0\in \Int(C)$ is given, and $p$ is the Minkowski functional (the gauge) associated 
with $C-x_0$ (this is a consequence of  \cite[Theorem 2.21, Remark 2.22(a), p. 27]{VanTiel1984}  
and \cite[Theorem 2.27, p. 29]{VanTiel1984}). 

Yet another example, which is related to the second one mentioned above, is the 
case of star bodies. These objects are important ones in the geometry 
of numbers theory \cite{Cassels1997,DodsonKristensen2006,GruberLek,Mahler1946,Mahler1946II,Minkowski1911,Mordell1945}.  
A set $K\subseteq \R^n$ is called a star body whenever there exists a  
function $f:\R^n\to \R$, called a distance function, such that 
$K=f^{\leq 1}$ and $f$ has the following properties: 
it is continuous, nonnegative, does not coincide with the zero function, and finally, 
for all $t\geq 0$ and $x=(x_k)_{k=1}^n\in \R^n$ the equality $f(tx)=tf(x)$ holds. 
(Sometimes one can find slightly different formulations of the definition 
of star bodies, such as in \cite{Mahler1946} where it is 
assumed that $f(tx)=|t|f(x)$ for all $t\in\R$ and $x\in\R^n$. However, in 
common scenarios all of these formulations are essentially equivalent.) 
Typical examples of distance functions which illustrate the richness of this 
class of functions are 
\begin{equation*}
f(x)=C\left(\prod_{k=1}^n |x_k|^{p_k}\right)^{1/\sum_{k=1}^n p_k}\quad \textnormal{and}\quad 
f(x)=\left|\sum_{k=1}^r C_k|x_k|^{p}-\sum_{k=r+1}^n C_k|x_k|^{p}\right|^{1/p} 
\end{equation*}
and multiplications, additions or subtractions of such functions (with suitable 
powers and absolute values). Here $C,p,p_1,\ldots,p_n$ are given positive numbers, $C_1,\ldots,C_k$ are non-negative 
numbers such at least one of them is positive,  and $r$ is a given integer in $[1,n]$ 
(and $\sum_{k=r+1}^n p_k:=0$ when $r=n$). If, as in \cite{Siegel1989}, 
only distance functions obtained from norms are considered, then the 
corresponding star body is bounded. However, in general it may not be bounded as in 
the case where the distance function is $f:\R^2\to\R$ defined 
by $f(x_1,x_2)=|x_1^2-x_2^2|^{1/2}$. 

The subset $f^{=1}$ is called the boundary of $K$ and the subset $f^{<1}$ is called 
the interior of $K$. These subsets play an important role in the theory of the geometry of numbers 
in the context of evaluating the number of lattice points of certain lattices (critical lattices) 
in certain regions. It turns out that they coincide with their topological colleagues, 
and moreover, $f^{\leq 1}=\overline{f^{<1}}$. See, e.g., \cite[pp. 105-107]{Cassels1997} for 
the simple proof of this claim and related ones. 
These properties can be generalized to star bodies contained in topological vector spaces.

\subsection{The phenomenon discussed in this paper: }\label{subsec:IntroductionPhenomenon}  This paper discusses a phenomenon 
similar to the one described above, where now the set $S$ has the form of a Voronoi cell 
(or a dominance region) of a given site $P$ with respect to another given 
site $A$, namely 
\begin{equation}\label{eq:dom}
S=\{x\in X: d(x,P)\leq d(x,A)\}=:\dom(P,A).
\end{equation}
The sites $P$ and $A$ are nothing but nonempty subsets contained in a convex subset $X$ 
which is contained in a normed space, $d(x,P)=\inf\{d(x,p): p\in P\}$,  
and $d$ is the distance function induced by the norm. The sites 
are assumed to be positively separated, that is, 
\begin{equation}\label{eq:d(P,A)}
d(P,A):=\inf\{d(p,a): p\in P, a\in A\}>0.
\end{equation}
The set $S$ (which is nonempty since $P\subseteq S)$ can be represented  in the form \eqref{eq:0Level} 
where $f:X\to \R$ is defined by $f(x)=d(x,P)-d(x,A)$ for all $x\in X$, but in general this function is not convex 
and hence one cannot conclude in advance that the above mentioned phenomenon 
holds for it (see also the paragraph after the next one). 

 The Voronoi cell is the basic component in what is known as the Voronoi diagram, 
 a geometric structure which appears in many fields in mathematics, science, and technology 
and has diverse applications \cite{Aurenhammer,ConwaySloane,CSKM2013,VoronoiCVD_Review,VoronoiWeb,GruberLek,OBSC}. 
Given a tuple of nonempty sets $(P_k)_{k\in K}$ 
called the sites or the generators, the Voronoi cell (or Voronoi region) of 
the site $P_k$ is the set of all points in the space 
whose distance to $P_k$ is not greater than their distance to the other sites $P_j,j\neq k$. In other 
words, the Voronoi cell of $P_k$ is nothing but the dominance region $R_k=\dom(P_k,A_k)$ where 
$A_k=\cup_{j\neq k}P_j$. The Voronoi diagram is the tuple of Voronoi cells $(R_k)_{k\in K}$. 
See Figures \ref{fig:BoundaryInteriorClassicalVoronoi6Site}-\ref{fig:BoundaryInterior4Site2PerSiteL2718281828} 
for a few illustrations. 
Voronoi diagrams have been the subject of an investigation for more than 160 years, 
starting formally with Dirichlet \cite{Dirichlet} and Voronoi \cite{Voronoi1908} in the 
context of the geometry of numbers and informally with Descartes (astronomy) or even before. 
They have been considerably investigated during the last 40 years, 
but mainly in the case of point sites in finite dimensional Euclidean spaces 
(frequently only in $\R^2$ or $\R^3$). Not much is known about them in other 
settings, e.g., in the case of non-Euclidean norms or sites having a general form 
(but some works studying them there do exist: see, e.g., the discussion below, the discussion 
in Example \ref{ex:ParallelSegment},  
and some references mentioned in \cite{ReemVorStabilityNonUC2012}).   
Papers studying them  in an infinite dimensional setting seem to be 
rare at the moment \cite{KopeckaReemReich,ReemISVD09,ReemGeometricStabilityArxiv}. 

The main result of this paper is Theorem \ref{thm:InteriorBoundary} which establishes 
that under the above mentioned assumptions on the sites, if the normed space belongs to the 
wide class of uniformly convex spaces (see Section \ref{sec:UniConvTraingle}), then 
\begin{equation}
\partial(\dom(P,A))=\{x\in X: d(x,P)=d(x,A)\},\label{eq:boundary}
\end{equation}
\begin{equation}
\Int(\dom(P,A))=\{x\in X: d(x,P)<d(x,A)\},\label{eq:interior} 
\end{equation}
\begin{equation}
\dom(P,A)=\overline{\{x\in X: d(x,P)<d(x,A)\}},\label{eq:closure}
\end{equation}
where $\partial (S)$, $\Int(S)$, and $\overline{S}$ are respectively the boundary, interior, and closure of the 
set $S$ (with respect to $X$). The property expressed in \eqref{eq:boundary}-\eqref{eq:closure} 
 may seem intuitively clear at first glance, and even 
true in any metric space, but simple counterexamples show that this is not the case even in $\R^2$ 
(see Example ~\ref{ex:ParallelSegment}, Figure ~\ref{fig:BoundaryInteriorLargeBisectorSmallStrictInequality}, 
and Remark ~\ref{rem:f_closure_boundary}). 
The set $\{x\in X: d(x,P)=d(x,A)\}$ is called the bisector of $P$ and $A$, and the property 
expressed in \eqref{eq:boundary} implies that a bisector of two positively separated sites 
in a uniformly convex space cannot be ``fat'', that is, it cannot contain 
a ball (see Figure \ref{fig:BoundaryInteriorLargeBisectorSmallStrictInequality} for a counterexample). 

In the case of finite dimensional uniformly convex spaces
 (i.e., finite dimensional strictly convex spaces) 
 \eqref{eq:boundary}-\eqref{eq:closure} are essentially known, at least in 
 the 2-dimensional case with certain sites (e.g., points or polygonal sets), 
 but it seems that only recently these properties have been established, 
 in a closely related formulation,  in the  case of general dimension and general 
 closed sites  \cite[Lemma 6]{ImaiKawamuraMatousekReemTokuyamaCGTA}. A somewhat 
 related discussion \cite{Horvath2000} in the case of point sites in a finite dimensional 
 real strictly convex normed space implies that the induced bisector is homeomorphic to a hyperplane. 
 When the space is not strictly convex, then  \eqref{eq:boundary} does not necessarily hold 
 as is known for a long time, since the bisectors can be fat or strange: See the discussion 
 in Example \ref{ex:ParallelSegment}. In this connection it is interesting to note that it is 
 known for a long time that  the convexity of the Voronoi cell of two point sites $P$ and $A$ 
 (or the fact that the bisectors are hyperplanes) characterizes  the Euclidean norm 
  \cite{Day1947,Gruber1974I,Gruber1974II,Mann1935,Woods1969}. This can be generalized 
 to the case of Voronoi cells induced by lattices \cite{Gruber1974I,Gruber1974II,Mann1935}. 
 See also \cite{MartiniSwanepoel2004} regarding related properties of bisectors induced by 
 two point sites in finite dimensional normed spaces.

Another result related to  \eqref{eq:boundary}-\eqref{eq:closure} is 
\cite[Lemma 5.1, Lemma 5.2]{KopeckaReemReich}. The setting  
now is any (possibly infinite dimensional) uniformly convex space. In its simple 
version the result says that $\dom(P,A)$ is homeomorphic to a closed and bounded 
convex set whenever $P=\{p\}$ where $p$ is a point contained in the (relative) 
interior of $X$ and whenever $X$ is closed and bounded. 
This generalizes the well-known fact that in the Euclidean norm the Voronoi cell of a point site 
is convex. One cannot expect to generalize the above in a naive way to the case where $P$ is general, 
 since now   $\dom(P,A)$ is not necessarily connected (see Figure \ref{fig:BoundaryInterior4Site2PerSiteL2718281828}). 
 
\subsection{Issues related to the proof: }  The main difficulty in trying to 
establish \eqref{eq:boundary}-\eqref{eq:closure} 
 in the general case of infinite dimensional spaces and general sites $P$ and $A$ is 
 the fact 
that the distance between a point and a set is not necessarily attained even if the set is closed. 
The reason why this property is so helpful is because it gives one candidates for satisfying some 
desired conditions, candidates which are absent in the general case. The idea is 
to look for the candidates along a certain interval, somewhat similarly to the case of 
proving that the closure of a the (nonempty) interior of a convex set is the closure of 
the set itself (Subsection \ref{subsec:Background}). 
More specifically, if $z\in X$ is on the bisector of $P$ and $A$, i.e., 
it satisfies the equation $d(z,P)=d(z,A)$, then because of \eqref{eq:boundary} 
one needs to show in particular that in every 
neighborhood of $z$ there are points outside $\dom(P,A)$, i.e., points $x\in X$ satisfying $d(x,A)<d(x,P)$. 
If $d(z,A)$ is attained at some point $a\in A$, then one may guess that the needed points $x$ can be taken 
from the segment $[a,z)$ near the endpoint $z$. It turns out that this is indeed true because of 
the uniform (actually strict) convexity of the space as shown in \cite[Lemma 6]{ImaiKawamuraMatousekReemTokuyamaCGTA} 
and this is true even if $d(P,A)>0$ is weakened to $P\cap A=\emptyset$. 
Unfortunately, as mentioned above, in infinite dimensional spaces the distance between a point and a general set is not 
necessarily attained.

The way the above mentioned difficulty is treated here 
is by using the assumption that $d(P,A)$ is positive and by using a nice and interesting improvement of the triangle 
inequality, due to Clarkson \cite[Theorem 3]{Clarkson}. This strong triangle 
inequality (as called in \cite{Clarkson}) allows one, 
after a suitable selection of certain parameters,  to derive explicit geometric estimates
 which show that if the above mentioned point $z$ is in the interior of $\dom(P,A)$, 
 then a contradiction must occur. 
Interestingly, although the strong triangle inequality was formulated in the 
very famous paper \cite{Clarkson} of Clarkson, it has been almost totally forgotten (in contrast to 
the well-known Clarkson's inequalities for $L_p$ spaces \cite[Theorem 2]{Clarkson}). Actually, 
in spite of a comprehensive search the author has made, evidences to its existence were found only 
in \cite{Clarkson} and later in  \cite{Plant,SmithTurett1990}. It will not be surprising 
if additional references will be found, but it seems that this strong triangle 
inequality is far from being a mainstream knowledge. 
Recently this inequality has been used for proving 
another property of Voronoi cells, namely their geometric stability with 
respect to small changes of the sites \cite{ReemGeometricStabilityArxiv} (see also 
Subsection \ref{subsec:Applications}). The derivation of 
\cite[Lemma 5.1, Lemma 5.2]{KopeckaReemReich} mentioned above is based indirectly on the strong 
triangle inequality via some of the results established in  \cite{ReemGeometricStabilityArxiv}.

\subsection{Possible applications: } \label{subsec:Applications} It may be of some interest to mention a few possible 
applications of the main result. 
One application is related to the geometric stability of Voronoi cells with 
respect to small changes of the sites. As shown in \cite{ReemGeometricStabilityArxiv}, 
a small perturbation of the sites, measured using the Hausdorff distance, yields 
a small perturbation in the Voronoi cells, measured again using the Hausdorff distance. 
In several applications related to Voronoi diagrams, such as in robotics \cite{SchwartzSharir1987}, 
the bisectors of the cells are important. Hence one may ask whether the bisectors are 
geometric stable under small perturbations of the sites. It turns out that  
 \eqref{eq:boundary} enables one to deduce this, assuming no neutral Voronoi region exists, 
 i.e., the union of the Voronoi cells is the whole space 
(this always holds when finitely many sites are considered and also holds in many scenarios involving 
infinitely many sites, such as the case of lattices; however, in general a neutral region can exist). 
The proof is essentially 
as the proof of \cite[Corollary 5.2]{ReemVorStabilityNonUC2012} (despite the somewhat different setting).  
This issue will be discussed in a revised version of \cite{ReemGeometricStabilityArxiv} which is 
is planned to be uploaded onto the arXiv soon. 

%
%

Another application of the main result is as an auxiliary tool in the proof that 
a certain iterative scheme involving sets converges to certain geometric objects. 
 These objects are variations of the concept of Voronoi diagram and formally 
 they are defined as the solution of a fixed point equation on a product space of sets. 
 The pioneering work of Asano, Matou\v{s}ek, and Tokuyama \cite{AMTn} (earlier announcements appeared in \cite{AMT2}) 
introduced and discussed an important member in this interesting family of objects. 
This object, called  ``a zone diagram'', was studied in \cite{AMTn} in the case of the Euclidean plane 
with finitely many point sites. An iterative scheme for approximating 
it was suggested there. 

Soon after \cite{AMTn}, in an attempt to better understand zone diagrams, 
the concept of ``a double zone diagram''  was introduced and studied 
in \cite{ReemReichZone} in a general setting ($m$-spaces: a setting which is more 
general than metric spaces). However, no way to 
approximate this object was suggested. Recently \cite{ReemZoneCompute} it has been shown that the algorithm 
suggested by Asano, Matou\v{s}ek, and Tokuyama converges in a rather general setting 
(a class of geodesic metric spaces which contains Euclidean spheres and finite dimensional uniformly convex spaces 
and infinitely many sites of a general form) to a double zone diagram, and sometimes also to a 
zone diagram. An important part in the proof was  to establish  
\eqref{eq:closure} in the  above mentioned setting. A careful inspection of the whole proof shows 
that in order to generalize the convergence to infinite dimensional uniformly convex normed spaces it is 
sufficient to prove \eqref{eq:closure} there and to make some 
modifications in certain additional auxiliary tools. This issue, which is a work in progress, 
will be discussed elsewhere. 
%
%

\subsection{Paper layout: } The paper is laid as follows. In Section \ref{sec:UniConvTraingle} 
 the concept of uniformly convex normed spaces is recalled and the 
 strong triangle inequality of Clarkson is presented.  The main result is established in 
 Section \ref{sec:UniformlyConvex}. In Section \ref{sec:Counterexamples}  a few 
  examples and counterexamples related to the main result are discussed. 
  Section \ref{sec:ConcludingRemarks} concludes the paper. 

\section{Uniformly convex spaces and the strong triangle inequality}\label{sec:UniConvTraingle} 
This section recalls the concept of uniformly convex normed spaces and presents the 
 strong triangle inequality of Clarkson. 
 
\begin{defin}\label{def:UniformlyConvex}\label{page:UniConvDef}
A normed space $(\widetilde{X},|\cdot|)$ is said to be uniformly convex if for each $\epsilon\in (0,2]$ there exists $\delta\in (0,1]$ such that for all $x,y\in \wt{X}$ satisfying $|x|=|y|=1$, if $|x-y|\geq \epsilon$, then $|(x+y)/2|\leq 1-\delta$.
\end{defin}
Typical examples of uniformly convex spaces are inner product spaces, the sequence spaces   $\ell_p$, the Lebesgue spaces $L_p(\Omega)$ ($1<p<\infty$), and a uniformly convex product of finitely many uniformly convex spaces. 
  The spaces  $\ell_1,\ell_{\infty},L_1(\Omega),L_{\infty}(\Omega)$ are typical examples of spaces which 
  are  not uniformly convex. See  \cite{BL2000,Clarkson,GoebelReich,LindenTzafriri,Prus2001} for more information.  

From the definition of uniformly convex spaces it is possible to obtain a function which 
assigns to the given $\epsilon$ a corresponding value $\delta(\epsilon)$. There are several  ways to obtain such a function,  but for the purposes of this paper $\delta$ should be increasing and to satisfy $\delta(0)=0$ and $\delta(\epsilon)>0$ for all $\epsilon\in (0,2]$. A familiar choice, which is not necessarily the most convenient one, is the modulus of convexity   
\begin{equation*}
\displaystyle{\delta(\epsilon)=\inf\{1-|(x+y)/2|: |x-y|\geq \epsilon,\,|x|=|y|=1\}}.\label{eq:delta}
\end{equation*}
For formulating the strong triangle inequality the definition of Clarkson's angle should be recalled. 
\begin{defin}\label{def:angle}
Given two non-zero vectors $x,y$ in a normed space, the angle (or Clarkson's angle, or the normed angle) $\alpha(x,y)$ between them is the distance between their directions, i.e., it is defined by
\begin{equation*}
\alpha(x,y)=\left|\frac{x}{|x|}-\frac{y}{|y|}\right|.
\end{equation*}
\end{defin}
\begin{thm}\label{thm:Clarkson}{\bf (Clarkson \cite[Theorem~3]{Clarkson})}
Let $x_1,x_2$ be two non-zero vectors in a uniformly convex normed space $(\widetilde{X},|\cdot|)$. If $x_1+x_2\neq 0$, then 
\begin{equation}\label{eq:ClarkTriangle}
|x_1+x_2|\leq |x_1|+|x_2|-2\delta(\alpha_1)|x_1|-2\delta(\alpha_2)|x_2|,
\end{equation}
where $\alpha_l=\alpha(x_l,x_1+x_2),\,l=1,2$.
\end{thm}
The original formulation of Clarkson's theorem is for finitely many non-zero terms 
(whose sum is not zero too) in 
a uniformly convex Banach space. An examination of the (simple) proof shows that the 
theorem actually holds in any normed space, not necessarily uniformly convex or Banach. 
However, it seems less useful in general normed spaces where it may happen that $\delta(\epsilon)=0$ 
even when $\epsilon>0$. Inequality \eqref{eq:ClarkTriangle} can obviously be extended to the case of zero terms by 
letting $\alpha(0,x):=0=:\alpha(x,0)$ for all $x$, but no use of this extension will be made here.

\section{The main result}\label{sec:UniformlyConvex}
In this section the main result (namely Theorem \ref{thm:InteriorBoundary} below) is proved. 
The proof is also based on  a simple lemma which is proved for the sake of completeness. 
Before stating both, here are a few words about the (standard) notation used below: 
$B(x,r)$ denotes the open ball with radius $r>0$ and center at $x\in X$; 
given points $a,b\in X$, the segments $[a,b]$ and $[a,b)$ denote the sets 
 $\{a+t(b-a): t\in [0,1]\}$ and $\{a+t(b-a): t\in [0,1)\}$ respectively; 
given a subset $S$ of $X$, its 
complement, closure, interior, boundary, and exterior (with respect to $X$) are  
 respectively $S^c$,$\overline{S},\Int(S),\partial (S)$ and $\Ext(S):=(\overline{S})^c$; 
 given a norm $|\cdot|$, the induced metric is $d(x,y)=|x-y|$. Given $f:X\to\R$, recall that 
$f^{\leq 0}=\{x\in X: f(x)\leq 0\}$ and  $f^{=0}=\{x\in X: f(x)=0\}$. Similarly  $f^{<0}$, 
$f^{\geq 0}$, and $f^{>0}$ are defined. 

\begin{lem}\label{lem:dom}
Let $(X,\tau)$ be a topological space let $f:X\to\R$.  If $f^{\leq 0}$ is closed and $f^{<0}$ is open, 
then:
\begin{enumerate}[(a)]
\item\label{item:Int} $f^{<0}\subseteq \Int(f^{\leq 0})$;
\item\label{item:Ext} $f^{>0}=\Ext(f^{\leq 0})$;
\item\label{item:boundary} $\partial(f^{\leq 0})\subseteq f^{=0}$;
\item\label{item:IntBoundary} equality holds in \eqref{item:Int} if and only  it holds in \eqref{item:boundary};
\item\label{item:Closure_dom(A,P)} if $\Int(f^{\geq 0})=f^{>0}$, then $f^{\leq 0}=\overline{f^{<0}}$;
\end{enumerate}
In particular, the above items hold when $f$ is continuous.
\end{lem}

\begin{proof}
\begin{enumerate}[(a)]
\item Since $f^{<0}$ is open by assumption and it 
is contained in $f^{\leq 0}$, the assertion follows from the fact that $\Int(f^{\leq 0})$ is the union of 
all the open subsets of $f^{\leq 0}$. 
\item $\Ext(f^{\leq 0}):=\left(\overline{f^{\leq 0}}\right)^c=\left(f^{\leq 0}\right)^c=f^{>0}$ by definition and 
because $f^{\leq 0}$ is assumed to be closed.
\item Suppose that $x\in \partial(f^{\leq 0})$. Then $x$ cannot belong to  
 $f^{<0}$ which is contained in $\Int(f^{\leq 0})$ by part \eqref{item:Int}, and cannot
belong to  $f^{>0}=\Ext(f^{\leq 0})$ by part \eqref{item:Ext}. Thus $x\in f^{=0}$.
\item Follows from $\partial(f^{\leq 0})\bigcup \Int(f^{\leq 0})=\overline{f^{\leq 0}}=f^{\leq 0}=f^{=0}\bigcup f^{<0}$  
and the fact that the terms in both unions are disjoint. 
\item Follows from $f^{<0}=(f^{\geq 0})^c$ and $(\Int(S))^c=\overline{S^c}$ for each $S\subseteq X$. 
\end{enumerate}
\end{proof}
\begin{remark}\label{rem:f_closure_boundary}
Neither the equality $f^{\leq 0}=\overline{f^{<0}}$ nor the equality 
$\partial(f^{\leq 0})=f^{=0}$ imply each other. A simple counterexample 
to the first case is obtained from the function $f:\R\to\R$ defined by 
$f(x)=-|x|$ if $|x|\leq 1$ and $f(x)=|x|-2$ when $|x|\geq 1$. The 
problem is with $x=0$. A counterexample to the second case is $g=-f$.  
\end{remark}

\begin{thm}\label{thm:InteriorBoundary} 
Let $X$ be a convex subset of a uniformly convex normed space $(\widetilde{X},|\cdot|)$.  Let 
$P,A\subseteq X$ be nonempty and suppose that $d(P,A)>0$. Then \eqref{eq:boundary}, \eqref{eq:interior}, 
and \eqref{eq:closure} hold.
\end{thm}
\begin{proof}
It is possible to write $\dom(P,A)=f^{\leq 0}$ where $f:X\to \R$ is the continuous function defined by $f(x)=d(x,P)-d(x,A)$ 
for each $x\in X$. 
It suffices to show that \eqref{eq:boundary} holds. Indeed, \eqref{eq:interior}  will be a obtained as a consequence of 
Lemma \ref{lem:dom}\eqref{item:IntBoundary}. But then, by noticing that obviously $d(A,P)>0$, it will be 
possible to obtain  \eqref{eq:boundary} and \eqref{eq:interior} with the roles of $P$ and $A$ exchanged, i.e., 
 with $g=-f$ instead of $f$. This and Lemma ~\ref{lem:dom}\eqref{item:Closure_dom(A,P)} (still applied to $f(x)=d(x,P)-d(x,A)$) 
will imply \eqref{eq:closure}. 

It will now be shown that \eqref{eq:boundary} holds. The  inclusion $\partial(\dom(P,A))\subseteq \{x\in X: d(x,P)=d(x,A)\}$ 
is implied by Lemma \ref{lem:dom}\eqref{item:boundary}. For the converse one, let $z$ be in the set $\{x\in X: d(x,P)=d(x,A)\}$.
Assume by way of contradiction that $z\notin \partial(\dom(P,A))$. It must be that $d(z,P)>0$, because otherwise 
$d(z,A)=d(z,P)=0$ and hence $d(P,A)=0$, a contradiction.  
Since $z\in \dom(P,A)\backslash\partial\dom(P,A)$ it follows that $z\in \Int(\dom(P,A))$. Thus there exists 
$\epsilon\in (0,d(z,P))$ such that the ball $B(z,\epsilon)$ is contained in $\dom(P,A)$. Let $\sigma\in (0,\infty)$ be 
arbitrary and let
\begin{equation}\label{eq:r}
r=\displaystyle{\min\left\{\sigma,\frac{d(P,A)}{4}, \frac{\epsilon}{2}\cdot\delta\left( \frac{d(P,A)}{4(\sigma+d(z,A))}\right)\right\}}. 
\end{equation}

For $r$ to be well defined the argument inside $\delta$ should be at most 2 (see page \pageref{eq:delta}). This is true without any 
assumption on $\sigma$. Indeed, given $\tau>0$ arbitrary, let $a'\in A$ and $p'\in P$ 
satisfy $d(z,a')<d(z,A)+\tau$ and $d(z,p')<d(z,P)+\tau$. The triangle inequality and the equality $d(z,P)=d(z,A)$ 
imply that $d(P,A)\leq d(p',a')\leq 2d(z,A)+2\tau$, and since $\tau$ was arbitrary this 
 implies that $d(P,A)/(4(d(z,A)+\sigma))<0.5<2$. 

Now let $a\in A$ and $p\in P$ satisfy 
\begin{equation}\label{eq:ap}
d(z,a)<r+d(z,A)\quad \text{and}\quad d(z,p)<d(z,P)+r/10. 
\end{equation}
From the choice of $a,p,z$, and $\epsilon$, 
\begin{equation}\label{eq:epsilon_z_p_a} 
\epsilon<d(z,P)\leq d(z,a)<r+d(z,A)=r+d(z,P)\leq r+d(z,p).
\end{equation}
By \eqref{eq:epsilon_z_p_a} the length of the segment $[a,z]$  is greater than $\epsilon$. 
Let $x\in [a,z]\subset X$ be such that $d(x,z)=\epsilon/2$. Then $x\in B(z,\epsilon)\subseteq \dom(P,A)$ and hence 
$d(x,P)\leq d(x,a)$. Let $q\in P$ satisfy $d(x,q)\leq d(x,P)+r/10$. By the above 
\begin{equation}\label{eq:dxqxa}
d(x,q)\leq d(x,P)+r/10\leq d(x,a)+r/10. 
\end{equation}
By the choice of $p$ and $\epsilon$ 
\begin{equation}\label{eq:d(x,z)d(z,q)}
d(x,z)=\epsilon/2<\epsilon<d(z,p)\leq d(z,P)+r/10\leq d(z,q)+r/10. 
\end{equation}
This and the fact that $\epsilon-r/10>\epsilon/2$ imply that $q\neq x$ and $q\neq z$. 
In addition, because of \eqref{eq:dxqxa}, $x\in [a,z]$, \eqref{eq:ap}, and \eqref{eq:r} it follows that 
\begin{multline}\label{eq:d(q,z)}
d(q,z)\leq d(q,x)+d(x,z)\leq d(a,x)+d(x,z)+r/10 \\ 
=d(z,a)+r/10<d(z,A)+\sigma+r/10. 
\end{multline}
For arriving at the desired contradiction distinguish between two cases. 

{\noindent\bf Case 1:} The angle $\alpha(z-x,z-q)$ satisfies the inequality  
\begin{equation}\label{eq:alpha}
\alpha(z-x,z-q)\geq d(P,A)/(4(\sigma+d(z,A))). 
\end{equation}
In this case by \eqref{eq:ap}, the strong triangle inequality \eqref{eq:ClarkTriangle}, by \eqref{eq:dxqxa}, 
 by $2d(z,x)=\epsilon$,  by \eqref{eq:alpha}, by the monotonicity of $\delta$, by $x\in [a,z]$, by \eqref{eq:r}, and by \eqref{eq:epsilon_z_p_a} it follows that 
\begin{multline*}
d(z,p)\leq d(z,P)+r/10\leq |z-q|+r/10\\
\leq r/10+|z-x|+|x-q|-2|z-x|\delta(\alpha(z-x,z-q))-2|x-q|\delta(\alpha(x-q,z-q))\\
\leq r/10+|z-x|+|x-a|+r/10-\epsilon\delta(d(P,A)/(4(\sigma+d(z,A))))\\
\leq |z-a|+r/5-2r\leq d(z,p)+r-9r/5<d(z,p),
\end{multline*}
a contradiction. All the angles are well defined because $z\neq x$, $z\neq q$ and $q\neq x$.

{\noindent \bf Case 2:} Inequality \eqref{eq:alpha} does not hold. Let $\theta=(q-z)/|q-z|$ and $\phi=(x-z)/|x-z|$. Then $q=z+|q-z|\theta$. Since $x\in [a,z]$ it follows that $a=z+|a-z|\phi$ and 
\begin{equation}\label{eq:phi_theta}
|\phi-\theta|=|(-\phi)-(-\theta)|=\alpha(z-x,z-q). 
\end{equation}
Let $s=d(z,a)-d(z,q)$. By \eqref{eq:d(q,z)} it follows that $s\geq -r/10$. 
By \eqref{eq:epsilon_z_p_a}, \eqref{eq:ap}, and $q\in P$ 
\begin{equation}
d(z,a)-r<d(z,p)\leq d(z,P)+r/10\leq d(z,q)+r/10. 
\end{equation}
This and \eqref{eq:r} imply that $s<11r/10\leq  11d(P,A)/40$. By combining this with $s\geq -r/10$ 
we see that $|s|\leq 11d(P,A)/40$. By this inequality, the definition of $s$, by \eqref{eq:d(q,z)}, by \eqref{eq:phi_theta}, 
since \eqref{eq:alpha} does not hold,  and since $r/10<\sigma+d(z,A)$, 
\begin{multline*}
|a-q|=\left|(z+|a-z|\phi)-(z+|q-z|\theta)|=|(s+|q-z|)\phi-(|q-z|\theta)\right|\\
\leq |s||\phi| +|q-z||\theta-\phi| < \frac{11d(P,A)}{40}+\frac{(d(z,A)+\sigma+r/10)d(P,A)}{4(\sigma+d(z,A))}\\
\leq d(P,A)(11/40+1/2)<d(P,A),
\end{multline*}
a contradiction because $q\in P$ and $a\in A$. This contradiction and the previous established one show 
that the assumption $z\not\in \partial (\dom(P,A))$ is false and prove \eqref{eq:boundary}  
(and \eqref{eq:interior}-\eqref{eq:closure}). 
\end{proof}

\section{Examples and Counterexamples}\label{sec:Counterexamples} 
This section presents a few examples and counterexamples related to Theorem ~\ref{thm:InteriorBoundary}.

\begin{expl}\label{ex:Illustrations}
Illustrations of Theorem \ref{thm:InteriorBoundary}  are given in 
Figures \ref{fig:BoundaryInteriorClassicalVoronoi6Site} and \ref{fig:BoundaryInterior4Site2PerSiteL2718281828}. 
In both figures the Voronoi diagrams of several sites $(P_k)_{k\in K}$ are presented and the corresponding 
dominance regions are the Voronoi cells $\dom(P_k,A_k)$, $A_k=\bigcup_{j\neq k}P_j$. In Figure \ref{fig:BoundaryInteriorClassicalVoronoi6Site} the setting is a square in $(\R^2,\ell_2)$ and each site is a point, and in Figure \ref{fig:BoundaryInterior4Site2PerSiteL2718281828} the setting is a square in $(\R^2,\ell_{p}),\,p\approx 2.71$ and each site has two points.  
\end{expl}
\begin{figure}[t]
\begin{minipage}[t]{0.45\textwidth}
\begin{center}
{\includegraphics[scale=0.72]{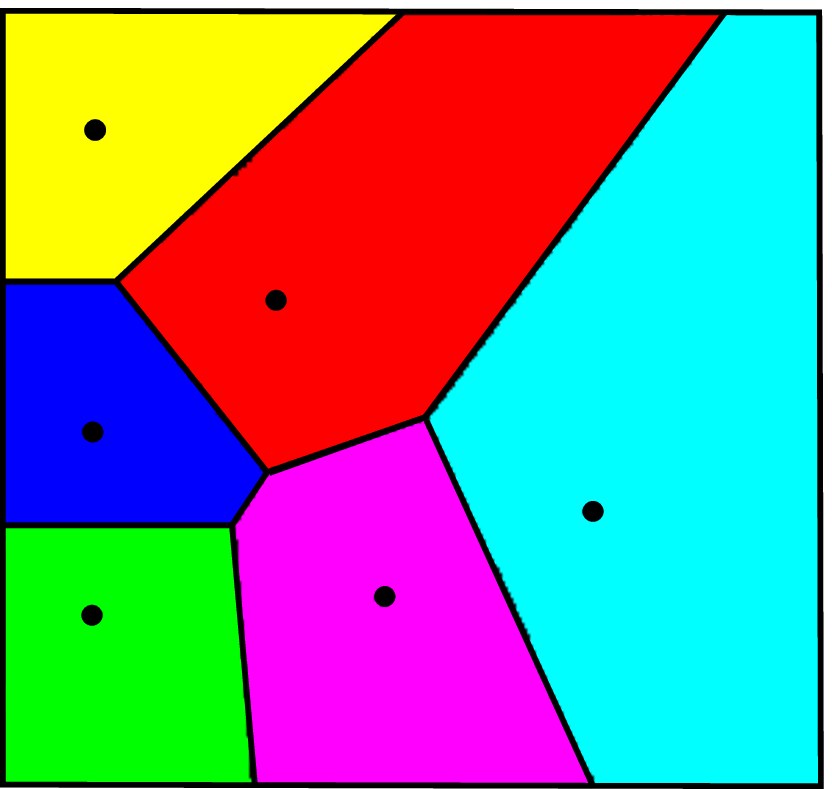}}
\end{center}
 \caption{The Figure of Example~ \ref{ex:Illustrations} (Euclidean norm).}
\label{fig:BoundaryInteriorClassicalVoronoi6Site}
\end{minipage}
\hfill
\begin{minipage}[t]{0.45\textwidth}
\begin{center}
{\includegraphics[scale=0.72]{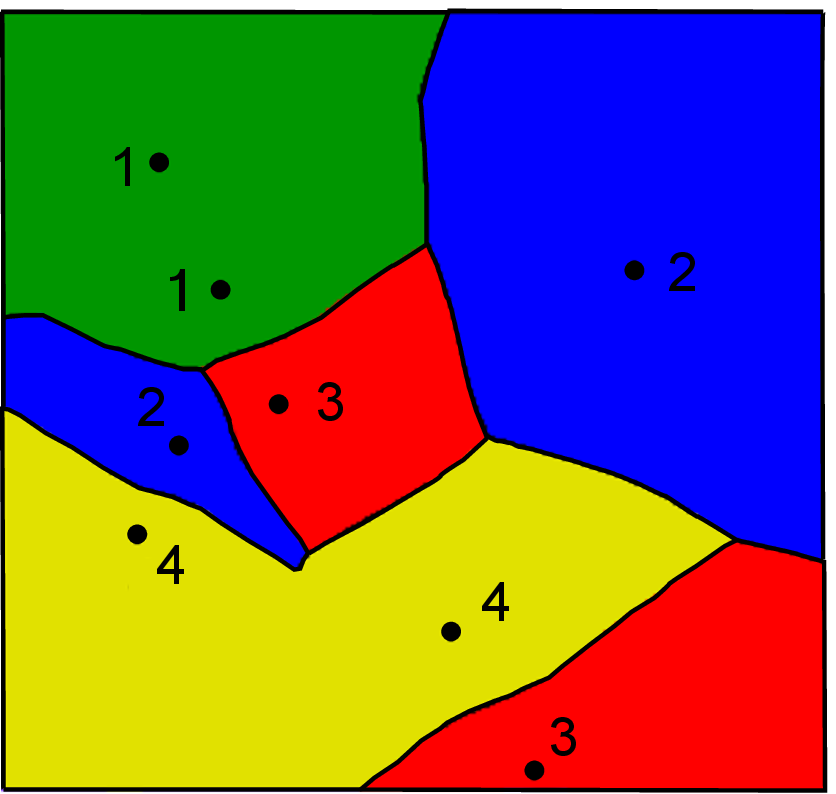}}
\end{center}
 \caption{The figure of Example ~\ref{ex:Illustrations} (the $\ell_p$ norm, $p\approx 2.71$).}
\label{fig:BoundaryInterior4Site2PerSiteL2718281828}
\end{minipage}
\hfill
\begin{minipage}[t]{0.45\textwidth}
\begin{center}
{\includegraphics[scale=0.72]{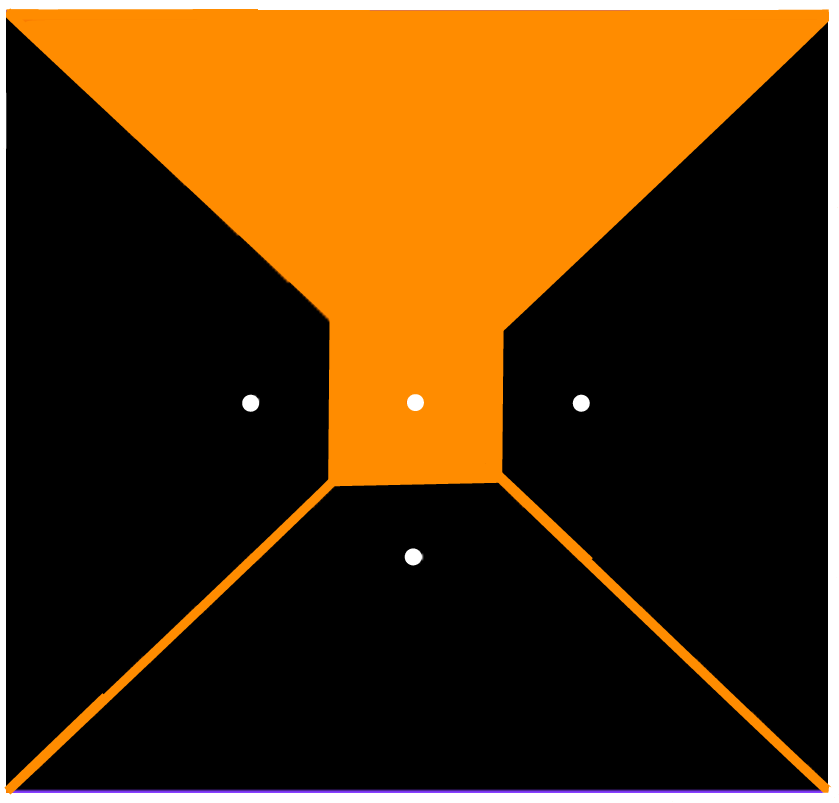}}
\end{center}
 \caption{The setting of Example ~\ref{ex:ParallelSegment}. The shown shape is the 
 Voronoi cell of $P=\{(0,0)\}$ with respect to  $A=\{(-2,0),(2,0),(0,-2)\}$, in a square in 
 $(\R^2,\ell_{\infty})$.}
\label{fig:BoundaryInteriorLinftyCell0}
\end{minipage}
\hfill
\begin{minipage}[t]{0.45\textwidth}
\begin{center}
{\includegraphics[scale=0.72]{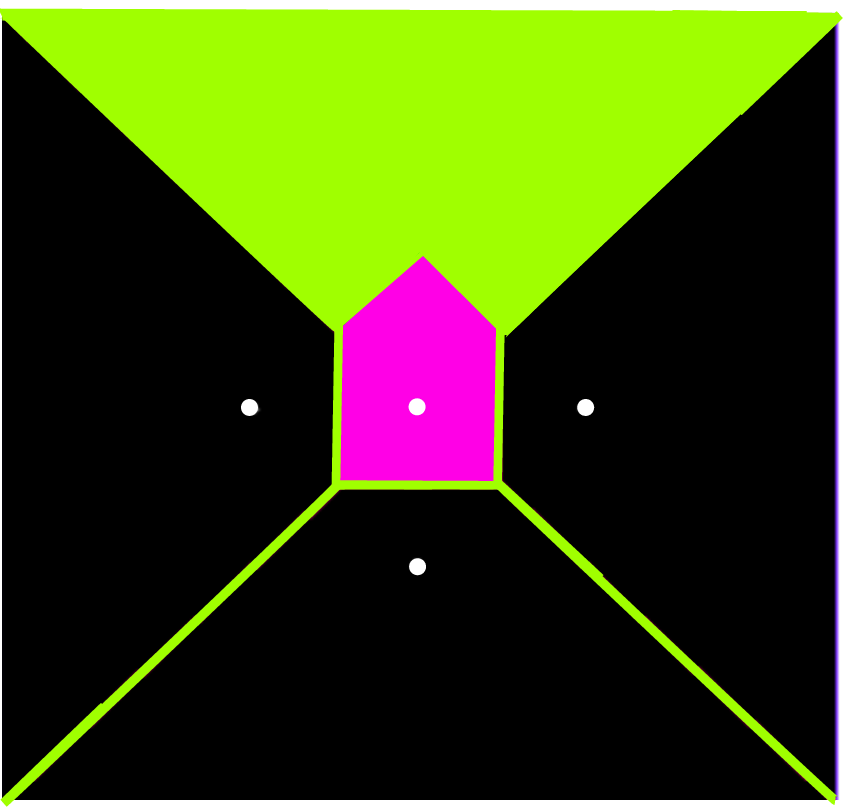}}
\end{center}
 \caption{The setting of Figure \ref{fig:BoundaryInteriorLinftyCell0} and Example \ref{ex:ParallelSegment}. 
 The set of strict inequality is the small purple ``house'' (pentagon) in the middle. The set of equality (the full green 
 ``W'', the sides of the house, and the rays)  contains a large part of the interior.} 
\label{fig:BoundaryInteriorLargeBisectorSmallStrictInequality}
\end{minipage}
\end{figure}

\begin{expl}\label{ex:ParallelSegment}
When the space is not uniformly convex, then the conclusion of Theorem ~\ref{thm:InteriorBoundary} does not 
necessarily hold. 
Indeed, consider $\dom(P,A)$ where $P=\{(0,0)\}$ and $A=\{(-2,0),(2,0),(0,-2)\}$, 
in the square $[-5,5]^2$ in $(\R^2,\ell_{\infty})$. See Figure \ref{fig:BoundaryInteriorLinftyCell0}. 
The set of strict inequality is relatively small (the ``house'' around $P$) compared to the 
set where equality holds (the bisector). This latter set contains a large part of the interior. 
See Figure \ref{fig:BoundaryInteriorLargeBisectorSmallStrictInequality}. 
In fact, when $X=\R^2$, then the former set remains the same (and hence bounded) 
while the latter grows and it is not bounded. A closely related example is 
 \cite[Example 4]{Horvath2000}. The setting there is the lattice of point sites 
 generated by the vectors $(2,0), (0,8)$ in $(\R^2,\ell_{\infty})$ 
and the considered cell is of the site $P=\{(0,0)\}$. In this case  $A$ is the set of all other sites. 
The resulting Voronoi cell is bounded and it is the union of a small middle hexagon (strict inequality) 
and large concave pentagons (bisector). The  difference between the example of 
Figure ~\ref{fig:BoundaryInteriorLargeBisectorSmallStrictInequality} (which was discovered before the 
author become aware to  \cite[Example 4]{Horvath2000}) and  \cite[Example 4]{Horvath2000} is that in 
 \cite[Example 4]{Horvath2000} no rays appear as in the case of  Figure \ref{fig:BoundaryInteriorLinftyCell0}. 
Additional related examples can be found in \cite[p. 390, Figure 37]{Aurenhammer} 
($P=\{(-1,1)\}$, $A=\{(1,-1)\}$ in $(\R^2,\ell_1)$)  and \cite[p. 605, Fig. 1(b)]{Lee}, \cite[p. 191, Figure 3.7.2]{OBSC} 
($P=\{(-1,-1)\}$, $A=\{(1,1)\}$ in $(\R^2,\ell_1)$) where, for instance, the bisector in the first case is 
$\left((-\infty,-1]\times [1,\infty)\right)\cup \{(t,-t): t\in [-1,1]\}\cup \left([1,\infty)\times (-\infty,-1]\right)$. 
Now the set of strict inequality is not bounded. 
%
%

\end{expl}

\begin{expl}\label{ex:NotAttained} 
This example shows that the condition $d(P,A)>0$ in Theorem \ref{thm:InteriorBoundary} cannot be weakened to 
 $P\cap A=\emptyset$ without further assumptions.  Indeed,  let $(\wt{X},|\cdot|)$ be the infinite 
  dimensional Hilbert space $\ell_2$. Let $X=\wt{X}$ and  
\begin{equation*} 
P=\{e_1\}\cup \{((n+1)/n)e_{n}: n=2,3,4,\ldots\}, \quad A=\{((n+2)/n)e_{n}: n=2,3,4,\ldots\},
\end{equation*} 
where $e_n$ is the $n$-th element in the standard 
basis, i.e., its $n$-th component is 1, and the other components are 0. 
For $z=0$ the equality $1=d(z,e_1)=d(z,P)=d(z,A)$ holds. 
 However, $z$ is in the interior of $\dom(P,A)$ since a simple 
check shows that the ball $B(z,0.1)$ is contained in $\dom(P,A)$. Thus \eqref{eq:boundary} 
does not hold. 

\end{expl}

\begin{expl}\label{ex:PA_Not_Disjoint}
This example shows that if $P\cap A\neq \emptyset$, then Theorem \ref{thm:InteriorBoundary} may be  
violated even in the case of a 2-dimensional space (in contrast to the case where $P\cap A=\emptyset$ 
as mentioned in Subsection \ref{subsec:IntroductionPhenomenon}). 
Indeed, let $(\wt{X},|\cdot|)$ be $\R^2$ with the Euclidean norm and let $X=\R^2$. 
Let $P=\{(-10,0),(0,0)\}$ and $A=\{(0,0), (10,0)\}$. Let $S=[-1,1]\times \R$. 
Then $S\subset [-2,2]\times \R \subseteq \dom(P,A)$. Thus $S\subset \Int(\dom(P,A))$. But 
$d(z,P)=d(z,A)=d(z,(0,0))$ for each $z\in S$. Therefore \eqref{eq:boundary} does not hold. 
\end{expl} 

\section{Concluding remarks}\label{sec:ConcludingRemarks}
It may be of interest to further investigate the phenomenon described in 
this note in various domains of mathematics and to find interesting 
applications of it. Perhaps a discontinuous version related to the phenomenon 
can be formulated (a simple example where this holds: 
let $C=f^{\leq 0}$ where $f:\R^n\to\R$ is defined as $0$ on the boundary 
of the unit ball, arbitrarily positive outside the ball and arbitrarily 
negative inside the ball). This may help in the study of singularities 
of the boundary. 

Another possible direction for future investigation is to  weaken the assumption 
of uniform convexity to strict convexity 
(the unit sphere does not contain line segments but, in contrast to 
uniform convexity, now there is no uniform bound $\delta(\epsilon)>0$ 
on how much the midpoint $(x+y)/2$ should penetrate the unit ball 
assuming $|x|=|y|=1$ and $|x-y|\geq \epsilon$). 
We conjecture that in this case there are 
counterexamples to Theorem ~\ref{thm:InteriorBoundary}. 
Alternatively, one may try to work with general normed spaces (under additional 
assumptions on the sites) or with spaces which are not linear. 
As a matter of fact, recently 
\cite{ReemVorStabilityNonUC2012,ReemZoneCompute} certain related 
results have been obtained. In the first paper (See Section 7 in the 
current arXiv version) a closely related 
 result (Lemma 9.11) is used  as a tool for proving the 
geometric stability of Voronoi cells with respect to small changes 
 of the sites in normed spaces which are not uniformly convex,  
 under some assumptions on the relation between the structure of 
 the unit sphere and the configuration of the sites. 
In the second paper (see Section 7) again a closely related  result is used for proving the 
convergence of an iterative scheme for computing a certain geometric 
object in a class of geodesic metric spaces.  
However, in both cases the distance between any  point in the space 
and both sites $P$ and  $A$ is assumed to be attained and hence the case 
of arbitrary sites in an infinite dimensional setting is not in the scope 
of these results.  

Finally, studying sets represented by a system of inequalities instead of 
one inequality may be valuable, because, for instance, sets having this 
form appear frequently in optimization \cite{Bertsekas1999,BorweinLewis2006,Rockafellar1970}. 
In the case of Voronoi cells $R_k=\dom(P_k,\cup_{j\neq k}P_j)$ one observes that 
the cell is nothing but the sets of all points $x$ satisfying the system of inequalities 
$f_j(x)\leq0$ where $f_j(x)=d(x,P_k)-d(x,P_j)$ for all  $j\in K, j\neq k$. A simple check shows 
that $d(x,\cup_{j\neq k}P_j)=\inf\{d(x,P_j): j\neq k\}$ and hence, when $K$ is finite, 
one concludes from Theorem \ref{thm:InteriorBoundary} that $x\in\partial R_k$ if 
and only if $x$ satisfies the above system of inequalities and at least one inequality 
is equality, and $x\in \Int(R_k)$ if and only if $x$ satisfies the system of inequalities 
with strict inequalities. 
\bibliographystyle{amsplain}
\bibliography{biblio}

\end{document}